\documentclass[12pt,reqno] {amsart}
\usepackage{amsthm, mathtools, amsmath,amscd,graphicx,amssymb,
amstext, epsfig,verbatim, multicol,mathrsfs,mathbbol}

\newtheorem{thm}{Theorem}
\newtheorem{lem}{Lemma}
\newtheorem{prop}{Proposition}
\newtheorem{cor}{Corollary}
{\theoremstyle{remark}
\newtheorem {Rem}{Remark}}
{\theoremstyle{definition}
}
                         \newtheorem{defn}{Definition}
\def\div{\mathrm{div}}

\newcommand{\cl}{\textrm{cl}}
 \newcommand{\mQ}{\mathbb  Q}

\def\fchar{\mathrm{char}}

\newcommand{\mC}{\mathbb C}
                                                   \newcommand{\CC}{\mathcal C}

\newcommand{\vf}{\varphi}

                                             \newcommand{\alb}{\mathrm{alb}}
\newcommand{\beq}{\begin{equation}}
\newcommand{\eeq}{\end{equation}}

\title[Torsion Points of small order]{Torsion points of small order on cyclic covers of $\mathbb P^1$. II}
\author {Boris M. Bekker}
 \address{ St.Petersburg State University, 7/9 Universitetskaya nab., St. Petersburg, 199034 Russia.}
\email{ bekker.boris@gmail.com}
\author {Yuri G. Zarhin}
\address{Pennsylvania State University, Department of Mathematics, University Park, PA 16802, USA}
\email{zarhin@math.psu.edu}
\thanks{The second named author (Y.Z.) was partially supported by Simons Foundation Collaboration grant   \# 585711  and the Travel Support for Mathematicians Grant MPS-TSM-00007756 from the Simons Foundation. Most of this work was done in August-September 2025 during his stay at the Max-Planck-Institut f\"ur Mathematik (Bonn, Germany), whose hospitality and support are gratefully acknowledged.}

\begin{document}
\begin{abstract}
Let $d\geq 2$ be an integer, $K_0$  a perfect field such that $\fchar(K_0) \nmid d$,  $n > d$ an integer prime to $d$, $f(x)\in K_0[x]$ a degree $n$ monic polynomial without repeated roots,  and $\mathcal{C}_{f,d}$ a smooth projective  model of  the affine curve $y^d=f(x)$. Let $J(\mathcal{C}_{f,d})$ be the Jacobian of the $K_0$-curve $\mathcal{C}_{f,d} $. We identify $\mathcal{C}_{f,d}$ with its canonical image in $J(\mathcal{C}_{f,d})$ (such that the infinite point of $\mathcal{C}_{f,d}$ goes to the zero of the group law on $J(\mathcal{C}_{f,d})$).

We say that an integer $m>1$ is $(n,d)$-reachable over $K_0$ if there  exists a polynomial $f(x)$ as above such that $\mathcal{C}_{f,d}(K_0)$ contains a torsion point of order $m$.

Earlier we proved  that if $m$ is  $(n,d)$-reachable, then either $m=d$ or $m \geq n$ (in addition, both $d$ and $n$ are  $(n,d)$-reachable).

In the present paper we prove the following.

\begin{itemize}
\item
 If $n<m<2n$ and if $m$ is $(n,d)$-reachable  over $K_0$, then either $d|m$ or $m \equiv n \bmod d$.
 \item
If  either
$\fchar(K_0)=0$ or $K_0$ in infinite and
$\fchar(K_0)>n$, then  $d\cdot  [(n+d)/d]$ is $(n,d)$-reachable  if and only if $n-(d-1)\cdot [(n+d)/d]\ge 0$.
\item
If $\fchar(K_0)=0$, then $n+d$ is $(n,d)$-reachable  if and only if $d^2-2d<n$.
\item
If $d=2$ (the hyperelliptic case) and $\fchar(K_0)=0$,  then $m$ is  $(n,d)$-reachable  if $n+1 \le m \le 2n+1$.
(The case when $n \le m \le 3(n-1)/2$ was done earlier by E.V. Flynn.)
\end{itemize}

\end{abstract}
\subjclass[2010]{14H40, 14G27, 11G10, 11G30}

\keywords{Cyclic covers, Jacobians, torsion points}
\dedicatory{Dedicated to Leonid Makar-Linanov on the occasion of his 80th birthday}
\maketitle
\section{Introduction}
Let $K$ be an algebraically closed field,
 $\mathcal C$ a smooth irreducible projective curve of positive genus $g$ over $K$, and $J(\mathcal C)$ the jacobian variety of $\mathcal C$.
For any point $O\in \mathcal C(K)$, the map $P\mapsto \cl((P)-(O))$, where $\cl((P)-(O))$ is  the linear equivalence class of the degree zero  divisor $(P)-(O)$,
gives rise to an embedding $\alb_O:\mathcal C\to J(\mathcal C)$ sending $O$ to the zero element of $J(C)$.

 We identify $\mathcal C$ with its image $\alb_O(\mathcal C)\subset J(\mathcal C)$. Let $m$ be a positive integer. We
 say that a point $P$ of $\CC(K)$ is a torsion point of order $m$ if its image $\alb_O(P)$
has order $m$ in the group $J(\CC)(K)$. Clearly, $O$ is the only point of order $1$.

Manin and Mumford conjectured and  Michel Raynaud \cite{Raynaud} proved that the set of torsion points in $\CC(K)$ is finite if $\fchar(K)=0$ and $g>1$.
(See also \cite{Ribet,Tsermias,Boxall}.)

It is  natural  to ask which values can take $m$ for a fixed embedding $\alb_O$ of $\CC$.
More precisely:
\begin{itemize}\item
    does  there exist a curve $\mathcal C$ having a   point of order $m?$

    \item if such a curve exists, then how many   points of order $m$ may $\mathcal C$ contain?
  \end{itemize}

In the present paper we consider the case  where $\CC$   is a cyclic cover of the projective line.
More precisely, let $d$ and $n$ be positive integers. In what follows we assume that $\CC=\CC_{f,d}$ is a smooth projective model of the plane affine smooth curve $C_{f,d}:y^d=f(x)$, where
\begin{equation} \label{restr1} 1< d< n, \quad (n,d)=1, \quad \fchar(K) \nmid d,\end{equation}
 $f(x) \in K[x]$ is  a degree $n$ polynomial without repeated roots, and $O=O_{f,d}$ is the {\sl infinite} point of $\CC_{f,d}$. Then $x$ and $y$ are rational functions on $\CC_{f,d}$, whose only pole is $O$ (see \cite{Gal}). The regular map
\beq\label{CC}\CC_{f,d}\setminus \{O\} \to C_{f,d}, \quad P \mapsto (x(P),y(P))\eeq
 is a biregular isomorphism of algebraic varieties over $K$. This implies that every rational function on $\mathcal C_{f,d}$ whose    only pole is $O$, can be represented as a polynomial in $x,y$ with coefficients in $K$.
 It is well known that  the genus $g$ of $\mathcal C_{f,d}$ is $(n-1)(d-1)/2$.  We identify $\CC_{f,d}\setminus \{O\}$ and $C_{f,d}$, using  isomorphism \eqref{CC}.
  In particular, we will talk freely about abscissas and ordinates of $K$-points on $\CC_{f,d}\setminus \{O\}$.

 It was proven by J. Boxall and D. Grant in \cite{Box1} that
for $g=2$ there are no points of order $3$ or $4.$ (See also their paper with  F.  Lepr\'evost \cite{Box2},  where points of order 5 on curves of genus 2 are discussed.)
  The second named author proved in  \cite[Theorem 2.8]{Zarhin} that   $\mathcal \CC_{f,2}(K)$  does {\sl not} contain a point of order $m$ if $g \ge 2$ and
 $3\leq m\leq 2g=n-1$.
These results
were generalized by the authors in \cite{BZ} to the case of hyperelliptic curves of arbitrary
genus $g\geq 2$ with torsion points of order $2g + 1$. In particular, we proved that if
$\fchar(K) = 2g + 1$, then $\mathcal C(K)$ contains at most $2$ points of order $2g + 1$.
(The case of $g=2, \ \fchar(K) =5$ was done earlier  in \cite{Box2}.)

 V. Arul \cite{Arul} listed all the torsion points on  Catalan curves $y^d=x^n+1$  and ``generic''   superelliptic curves in characteristic 0.
 (The case of `'generic'' hyperelliptic curves was done earlier by B. Poonen and M. Stoll \cite{PS}.)
In \cite[Th.1 and  Prop.1 and its Proof on p.10]{BZR} we proved the following assertion.

\begin{thm}
\label{thm0}
Suppose that $f(x)$ is a monic polynomial.
 Let $m>1$ be an integer and $P \in \CC_{f,d}(K)$ be a torsion point of order $m$.  Then:
 \begin{itemize}
  \item
  Either $m=d$ or $m \ge n$.
  \item
  $m=d$ if and only if $y(P)=0$, i.e., $x(P)$ is a root of $f(x)$.  In particular, the number of torsion points of order $d$ in $\CC_{f,d}(K)$  is $n$.
  \item
  If  $m \ne d$ and $P=(a,c)$, then $c \ne 0$ and  $\gamma P=(a,\gamma c)$ is also a point of order $m$ in $\CC_{f,d}(K)$
   for all $\gamma \in \mu_d$.
  In particular, the number of torsion points of order $m$ in $\CC_{f,d}(K)$  is divisible by $d$.
  \end{itemize}
\end{thm}

\begin{Rem}
 The assertions  of Theorem \ref{thm0} remain  true for arbitrary not necessarily monic  $f(x)$, see Remark \ref{monic} below.

\end{Rem}

We will need the following  useful corollary of Theorem
\ref{thm0}.

\begin{cor}
\label{O2n}
Let $m$ be a positive integer and $P \in \CC_{f,d}(K)$ be a torsion point.
Suppose that
$$P \ne O, \quad y(P) \ne 0,$$
and $m(P)-m(O)$ is a principal divisor on $\CC_{f,d}$.
Assume additionally that either $m$ is a prime, or $m<2n$, or $m$ is an odd integer that is strictly less than $3n$.
Then $P$ has order $m$.
\end{cor}

\begin{proof}

 Indeed, let $k$ be the order of $P$. Clearly,
 $$k>1,  \quad k |m.$$
 If $m$ is a prime, then $k=m$ and we are done. Suppose that $m$ is not a prime.

 Suppose that $m<2n$.
 If $k \ne m$, then
 $$k \le \frac{m}{2} <\frac{2n}{2}=n.$$
 So, $1<k<n$. By Theorem \ref{thm0}, $k=d$, which is not the case, since $y(P) \ne0$. The obtained contradiction implies that $m$ is an odd integer that is strictly less than $3n$.

 If $k<m$, then
 $$k \le \frac{m}{3}<\frac{3n}{3}=n.$$
 Hence, $k<n$ and, as above, $k=d$, which is not the case. The obtained contradiction proves that $k=m$ and we are done.

\end{proof}

Notice that  \cite[Theorem 1]{BZR} also contains the list of polynomials $f(x)$ such that the corresponding curve $\CC_{f,d}$ has a torsion $K$-point of order $n$. (See also \cite[Lemma 5.2.66]{Arul}.)

Let $K_0$ be a perfect subfield of $K$ and $f(x) \in K_0[x]$. Then the curve $\CC_{f,d}$ is naturally defined over $K_0$;
in addition,
$$O=O_{f,d} \in \CC_{f,d}(K_0)\subset \CC_{f,d}(K).$$

\begin{defn}
 Let $m>1$ be an integer. We say that $m$ is $(n,d)$-reachable over $K_0$
 if there exists a degree $n$ monic polynomial $f(x) \in K_0[x]$ without repeated roots such that
 the corresponding $\CC_{f,d}(K_0)$ has a point of order $m$.

\end{defn}

\begin{Rem}
\label{reach}
\begin{itemize}
\item[(i)]
If $K_1$ an overfield of $K_0$ and  $m$ is $(n,d)$-reachable over $K_0$, then it obviously
 $(n,d)$-reachable over $K_1$ as well. In particular, if $m$ is $(n,d)$-reachable over the field $\mQ$ of rational numbers
 if and only it is $(n,d)$-reachable over any field of characteristic zero.
 \item[(ii)]
First assertion of Theorem \ref{thm0} may be restated as an assertion that if $m$ is $(n,d)$-reachable over  $K$, then either $m=d$ or $m \ge n$. It follows from
\cite[Examples 1 and 2 on p. 11-12]{BZR} that $n$ is $(n,d)$-reachable over any $K_0$.
\item[(iii)] We claim that $m=d$ is $(n,d)$-reachable over any $K_0$. Due to Theorem \ref{thm0} it suffices to prove the existence of a monic degree $n$ polynomial  $f(x)\in K_0[x]$ without repeated roots that has a root in $K_0$.  If  $\fchar(K_0)=0$, then $f(x)=\prod_{i=1}^n (x-i)$
does the trick.
If $\fchar(K_0)=p>0$, then we may take $f(x)=x \cdot h(x)$, where
$h(x)\in \mathbb F_p[x]\subset K_0[x]$ is any degree $(n-1)$ monic polynomial
(with a nonzero constant term)
that is irreducible over  the prime finite field $\mathbb F_p$ of characteristic $p$.
\end{itemize}
\end{Rem}

If $K$ is the field $\mC$ of complex numbers and $K_0$ is the field of rational numbers $\mQ$ then results of E.V.  Flynn, F. Lepr\'evost, Q. Gendron
\cite{Flynn,Lep,Gendron}  for hyperelliptic curves (i.e., when $n$ is odd and $d=2$) may be restated as follows.

\begin{itemize}
\item
If  $m$ is an integer such that $n \le m \le 3(n-1)/2$ then $m$ is $(n,2)$-reachable over $\mQ$
(\cite[p. 435-436]{Flynn}.
\item
The number $(n-1)n$ is $(n,2)$-reachable over $\mQ$ (\cite[p. 53 and Lemme 3.3 on p. 56]{Lep}).
\item
If $m>n$ is an integer then it is $(n,2)$-reachable over $\mC$ (\cite[Th. 2.3]{Gendron}).
\end{itemize}

Here are main results of the present paper.

\begin{thm}[Main Theorem]
Let $K_0$ be a perfect subfield of $K$.
\label{mainT}
\begin{enumerate}
 \item[(1)]
 Let $m$ be an integer such that
$$1<m < nd.$$
Let us define the nonnegative integer $k:=[m/n]<d$, i.e., $k$ is the only integer such that
$$kn \le m<(k+1)n.$$
If $P$ is a point of order $m$ in $\CC_{f,d}(K)$ such that $1\le m<nd$,
then there exists a nonnegative $j\leq k$ such that $m\equiv jn\mod d$.

 \item[(2)]
 Let $m_0$ be the integer that lies strictly between $n$ and $n+d$ and is divisible by $d$, i.e.,
$$m_0:=d \cdot [(n+d)/d].$$
Let us put
$$\ell_0:=[(n+d)/d]=m_0/d.$$
If $n-m_0+\ell_0 <0$, then $m_0$ is not $(n,d)$-reachable  over $K$.
 \item[(3)]
 Let $m>n$ be an integer that  is divisible by $d$, i.e.,
 $\ell:=m/d$ is an integer.

 Assume that
  $$n-m+\ell \ge 0.$$
 If $K_0$ is infinite, then $m$ is $(n,d)$-reachable over $K_0$ if either $\fchar(K_0)=0$ or $\fchar(K_0)> n$.

\item[(4)]
Suppose that $\fchar(K_0)=0$. Then the following conditions are equivalent.
\begin{itemize}
 \item[(a)]
 $m_1:=n+d$ is $(n,d)$-reachable over $K_0$.
 \item[(b)]
 $d^2-2d<n.$
 \end{itemize}
\item[(5)]
Suppose that $\fchar(K_0)=0$. Let $e$ be a positive integer such that
$$ ed^2-(e+1)d<n.$$
 Then $n+ed$ is $(n,d)$-reachable over $K_0$.
 \end{enumerate}
 \end{thm}

 \begin{Rem}
 \label{div2n}
 Let $m$ be an integer such that
  $n<m<2n$.
  Suppose that $m$ is $(n,d)$-reachable over $K$.
 It follows from
 Theorem \ref{mainT}(1) that either $m$ is divisible by $d$ or $m \equiv n \bmod d$. In particular,
 \begin{itemize}
 \item
 every integer  that lies strictly between
 $n$ and $m_0 =d \cdot [(n+d)/d]$ is {\sl not} $(n,d)$-reachable over $K$;
 \item
 every integer  that lies strictly between $m_0$ and $m_1=n+d$  is {\sl not} $(n,d)$-reachable over $K$.
 \end{itemize}
 That is why it is natural to address the question (as   in Theorem \ref{mainT}(2,4))  whether $m_0$ and $m_1$ are $(n,d)$-reachable over $K$.
 \end{Rem}

The following assertion contains the result of Flynn quoted above (our proof is completely different).
 \begin{cor}
 \label{hyperelliptic}
 Suppose that $d=2$. Then every integer $m$ that lies between $n+1$ and $2n+1$ is $(n,2)$-reachable over any
 field $K_0$ of characteristic 0, including $K_0=\mQ$.
 \end{cor}

 \begin{proof}
 Since $(n,2)=(n,d)=1$, the integer $n$ is odd.
We have  $n+1 \le m \le 2n+1$. If $m$ is even, then $m \le 2n$ and
$$n-m+m/d = n-m/2\ge n-2n/2\ge 0.$$
Now the assertion follows from Theorem \ref{mainT}(3).

If $m$ is odd, then $m=n+2e$ with integer $e$  such that $1 \le e \le (n+1)/2$.
Then
$$ ed^2-(e+1)d=4e-2(e+1)=2e-2 \le 2 \cdot (n+1)/2 -2=n-1<n.$$
Hence, $ ed^2-(e+1)d<n$ and the result follows from Theorem \ref{mainT}(5).
 \end{proof}
The paper is organized as follows. Section \ref{prel} contains elementary results about divisors on $\CC_{f,d}$.
In  Section \ref{functions} we deal with rational functions on $\CC_{f,d}$  whose only pole is $O$,
and prove Theorem \ref{mainT}(1).
 In Section \ref{W^1}  we discuss torsion points, whose order is divisible by $d$.
In Section \ref{separ} we  prove Theorem \ref{mainT}(2,3). In Section \ref{order n+d points}
we discuss torsion points, whose   order is congruent to $n$ modulo $d$. Section \ref{powerSeries} contains auxiliary results about polynomials that we  use in Section
\ref{endMain},  proving  Theorem \ref{mainT}(4,5).

{\bf Acknowledgements}. We are grateful to the referee for useful comments.

 \section{Preliminaries}
\label{prel}
Let $d$ and $n$ be positive integers. In what follows we assume that conditions \eqref{restr1} are fulfilled.
Let $G=\mu_d$ be the multiplicative (order $d$) cyclic group of $d$th roots of unity in $K$.
Let \beq\label{f} f(x)=c_0\prod_{j=1}^n(x-w_j)\in K[x]\eeq
be a degree $n$  polynomial, where
$w_1,\ldots,w_n$ are distinct elements of $K$,   $c_0$ is a nonzero element of $K$, and $C_{f,d}: y^d=f(x)$ is the corresponding smooth plane affine curve over $K$.

  The projective closure $\bar C_{f,d}$  has exactly  one point at infinity, which may be singular. We  denote it by $\infty$. The group  $\mu_d$
 acts on $\bar C_{f,d}$ by
 $$(x,y) \mapsto (x, \gamma y) \quad \forall \gamma \in \mu_d;$$
 this action is fixing $\infty$.


Let $\mathcal C_{f,d}$ be the normalization of $\bar C_{f,d}$, which is a smooth projective model of $\bar C_{f,d}$. We have a map  $\Phi:\mathcal C_{f,d}\to \bar C_{f,d}$  which yields an isomorphism between $C_{f,d}$ and $\Phi^{-1}(C_{f,d})$.
Since $(n,d)=1$, $\Phi^{-1}(\infty)$ contains only one point $O$, and so the map $\Phi$ is bijective. In what follows we identify the points of $\mathcal C_{f,d}$ with the corresponding points of $\bar C_{f,d}$. The action of $\mu_d$ on $\bar C_{f,d}$ yields an action of $\mu_d$ on $\mathcal C_{f,d}$  such that the quotient $\mathcal C_{f,d}/\mu_d$ is isomorphic to the projective line $\mathbb P^1$. We have a cover $\mathcal C_{f,d}\to \mathbb P^1$ of degree $d$ with $n+1$ ramification points. By the Hurwitz formula,  the genus $g$ of $\mathcal C_{f,d}$ is $(n-1)(d-1)/2$.

\begin{Rem}
If $\fchar(K)\ne 2$ and $d=2$, then $n$ is odd and $\CC_{f,d}=\CC_{f,2}$ is  a   hyperelliptic  curve of genus  $g=(n-1)/2$, while $O$ is  one of the {\sl Weierstrass points} of $\CC_{f,2}$;
it is well known that the torsion points of order $2$ in $\CC_{f,d}(K)$ are  the remaining (different from $O$) $(2g+1)$ Weierstrass points on $\CC_{f,d}(K)$ (when $n>3$).

Notice that if $d>2$, then $\CC_{f,d}$ is {\sl not} hyperelliptic \cite[Lemma 2.2.3]{Arul}.
\end{Rem}

\begin{Rem}
 \label{monic}

 Recall that $c_0$ is the (nonzero) leading coefficient of our degree $n$ polynomial $f(x)\in K_0[x]$ and $O=O_{f,d}$ is the infinite point of $\CC_{f,d}$. Sometimes it is convenient to assume that $f(x)$ is monic, i.e., $c_0=1$. We can do it without loss of generality, using the following elementary observation.

 Choose nonzero integers $i$ and $j$   such that $d i+nj=1$ (recall that $n$ and $d$ are relatively prime. Then the polynomial
 $$h(x):= c_0^{-di} f(c_0^{-j} x) \in K_0[x]$$
 is a  degree $n$ polynomial without repeated roots, whose leading coefficient is
 $$c_0^{-di} \cdot c_0 \cdot  (c_0^{-j})^n=c_0^{-di+1 -nj}=c_0^{1-di-nj}=c_0^0=1.$$
 This means that $h(x)$ is a monic polynomial.
 Let us consider the corresponding plane affine smooth $K_0$-curve
 $$C_{h,d}: y^d =h(x).$$
 Then there is a biregular isomorphism of plane affine $K_0$-curves
 $$\Psi_0: C_{h,d} \to C_{f,d}, \quad (x,y) \mapsto (c_0^j x, c_0^{-i} y),$$
 which is $\mu_d$-equivariant and gives rise to the $\mu_d$-equivariant biregular isomorphism of smooth projective $K_0$-curves
 $$\Psi: \CC_{h,d} \to \CC_{f,d}, \quad \Psi(O_{f,d})=O_{h,d}.$$
 This implies that  $P\in \CC_{h,d}(K_0)$ is a torsion point of order $m$ if and only if $\Psi(P) \in \CC_{f,d}(K_0)$ is a torsion point of order $m$.
 On the other hand, if $Q$ is a $K$-point of $\CC_{h,d}$, then $y_1(Q)=0$ if and only if $y(\Psi(Q))=0$.

 This implies that an integer $m>1$ is $(n,d)$-reachable over $K_0$ if and only if there exists a degree $n$ not necessarily {\sl monic} polynomial $f(x) \in K_0[x]$ without repeated roots such that
 the corresponding $\CC_{h,d}(K_0)$ has a point of order $m$. In particular,
 the assertions of Theorem \ref{thm0} remain true for arbitrary not necessarily monic polynomials $f(x)$.
 \end{Rem}

We will need the following three elementary but useful assertions from \cite[Sect. 2, p. 4-5]{BZR}.

\begin{lem}\label{l1}
Let $\vf$ be a rational function on $\mathcal C_{f,d}$ defined by a polynomial $x-a$, where $a\in K$.
Let $P$ be an arbitrary $K$-point on $\CC_{f,d}$ with abscissa $a$,
i.e., $P$ is a zero of $\vf$ .
Then $$\div(x-a)
= \left( \sum_{\gamma \in \mu_d}(\gamma P)\right)-d(O).$$
\end{lem}

\begin{cor}\label{corlemma}
Let $\vf$ be a rational function on $\mathcal C_{f,d}$ defined by a polynomial \break $\vf =\prod_{i=1}^k(x-a_i)^{c_i},$  where
$c_i\in\mathbb Z$ and
$a_i\in K$
($a_i$ are not necessarily distinct). Then
\beq \div(\vf)=\sum\limits_{i=1}^k c_i\left(
 \left(\sum_{\gamma \in \mu_d}(\gamma P_i)\right)-d(O)\right).\eeq
If $D=(P_1)+\cdots  +(P_{r})$ is a sum of $r$ points, then
\beq\label{orbitdiv}
\sum_{\gamma \in \mu_d} \gamma D=
\sum_{i=1}^{r}\sum_{\gamma \in \mu_d}\gamma P_i.
\eeq
\end{cor}

\begin{lem}\label{l2}
Let $D$ be an arbitrary divisor of degree 0 on $\mathcal C_{f,d}$. Then the divisor $\sum_{\gamma \in \mu_d}\gamma D$ is principal.
\end{lem}

\section{Rational functions with  only pole at $O$}
\label{functions}

Let $\vf$ be a nonzero rational function on $\CC_{f,d}$ whose  only pole is $O$ and its multiplicity is  $M$.
 Then $\vf$ is a polynomial in $x,y$. Since $y^d=f(x)$, where $f(x)$ is a polynomial of degree $n$,  $\vf$ can be represented by a polynomial of the form

\beq \label{1} \vf=s_1(x)y^{d-1}+s_2(x)y^{d-2}+\cdots+ s_{d-1}(x)y+v(x), \,\text {where}\eeq
$$s_1(x),s_2(x),\ldots, s_{d-1}(x),v(x)\in K[x].$$
\vskip.3cm
Assume that at least one of the polynomials $s_i(x)$ is nonzero. Then
we can write equation \eqref{1} in the form
\beq \label{2} \vf=s_{\alpha_1}(x)y^{d-\alpha_1}+s_{\alpha_2}(x)y^{d-\alpha_2}+\cdots+s_{\alpha_k}(x)y^{d-\alpha_k}+v(x),\eeq
where
all polynomials $s_{\alpha_i}(x)$ are nonzero and all $\alpha_i$ are distinct positive less than $d$.
Since the function $x$ has at $O$ a pole of multiplicity $d$ and the function $y$ has at $O$ a pole of multiplicity $n$ the
multiplicity of the pole of  $s_{\alpha_i}(x)y^{d-\alpha_i}$ at $O$ equals \beq \label{polemult} d\deg (s_{\alpha_i}(x))+n(d-\alpha_i).\eeq
Assume that
$$d\deg (s_{\alpha_i}(x))+n(d-\alpha_i)=d\deg (s_{\alpha_j}(x))+n(d-\alpha_j),$$ where $i\neq j$.
Then
$$d(\deg (s_{\alpha_i}(x))-\deg (s_{\alpha_j}(x)))=n(\alpha_i-\alpha_j).$$
Consequently, if $i\neq j$, then  $\alpha_i\neq\alpha_j$
and   $\deg( s_{\alpha_i}(x))\neq \deg (s_{\alpha_j}(x))$.
Now the assumption $(n,d)=1$ leads to a contradiction. So all $s_{\alpha_i}(x)y^{d-\alpha_i}$ have distinct multiplicities at $O$.
If $v(x)$ is a constant, then
\beq\label{30}
M=\max\{d\deg (s_{\alpha_i})+n(d-\alpha_i)\}.
\eeq
If $v(x)$ is not a constant, then it has at $O$ a pole whose multiplicity is  $d\deg(v)$.
Since $d\deg (s_{\alpha_i}(x))+n(d-\alpha_i)$ is not divisible by  $d$,  we obtain that all  summands on the right-hand side of \eqref{2} have a pole  at $O$ with distinct multiplicities.
In  all the cases we have   that
\beq\label{3}
M=\max\{d\deg (s_{\alpha_i})+n(d-\alpha_i),d\deg (v)\}
\eeq
(which holds even if $v=0$).
Consequently, in all cases \beq\label{diopheq} M=td+kn,\eeq where $t$ is an integer and $k$ is an integer such that  $0\leq k<d$.

We have proved the following statement.

\begin{prop}\label{poleorder}
Let $\vf$ be a nonzero rational function on $\CC_{f,d}$ whose  {\sl only} pole is $O$  and its multiplicity is $M$. Assume that
$1<M<nd$ and an integer $k$ is such that $kn\leq M<(k+1)n$, where $0\leq k<d$. Then there exists a nonnegative
 integer $j\leq k$ such that $M\equiv jn\mod d$.
 \end{prop}

From Proposition \ref{poleorder} we immediately obtain the following statement (first assertion of Main Theorem).

\begin{thm}\label{order m}
Let  $Q$ be a $K$-point of order $m$ in $\CC_{f,d}$ and $1<m<nd$.
Let $k$ be an integer  such that $kn\leq m<(k+1)n$, where $0\leq k<d$. Then there exists a nonnegative
 integer $j\leq k$ such that $m\equiv jn\mod d$.
 In particular, if such an $m$ lies strictly between $n$ and $n+d$, then it is not congruent to $n$ modulo $d$ and therefore is divisible by $d$ and equals $d \cdot [(n+d)/d]$.
\end{thm}

\begin{Rem}
\label{Norm}
Let $P=(a,b) \ne O$ be a $K$-point of $\CC_{f,d}$,  $m>1$ an integer, and $\phi$ a nonzero rational function on $\CC_{f,d}$ such that
$$\div(\phi)=m(P)-m(O),$$
i.e., $P$ is the   only zero of $\phi$, $O$ is the   only pole of $\phi$, and they both have multiplicity $m$. As we have already seen, $\phi$ may be represented as $H(x,y)$, where  $H\in K[X,Y]$ is a polynomial in two variables   $X,Y$ with coefficients in $K$.  It follows  that for each $\gamma \in \mu_d$ the rational function $\phi\circ \gamma$ coincides with $H(x,\gamma y)$; in particular,
$$\div (H(x,\gamma y))=m(\gamma^{-1}P)-m(O) \quad \forall \gamma \in \mu_d.$$
This implies that the product
$$\mathrm{Norm}(\phi):=\prod_{\gamma \in \mu_d}H(x,\gamma y)$$
has divisor
$$m\left(\sum_{\gamma \in \mu_d}(\gamma^{-1} P)\right)-d m (O)=
m \left( \sum_{\gamma \in \mu_d}(\gamma P)-d (O)\right),$$
which,  in light of Lemma  \ref{l1}, coincides with
$$m \cdot \div(x-a)=\div((x-a)^m).$$
It follows that there exists a nonzero $A \in K$ such that
$$\mathrm{Norm}(\phi)=A(x-a)^m.$$

\end{Rem}

\section{Torsion points}\label{W^1}

 Let $P$ be a point of order $m>1$ on $\CC_{f,d}$. We have seen (Theorem~\ref{thm0}) that then either $m=d$ or $m\geq n$.
  Moreover, if $f(x)$ is monic, then the curve  $\CC_{f,d}$ has a $K$-point of order $n$ if and only if there exist $a \in K$
 and a polynomial $v(x) \in K[x]$ such that
 $$\deg(v) \le \frac{n-1}{d}, \quad f(x)=(x-a)^n+v^d(x)$$
 (see \cite[Th. 1(iii)]{BZR}).
 In this section we consider the case where $n<m<2n$. It follows from Theorem
 \ref{order m} that if $m$ is $(n,d)$-reachable over $K$, then either $d\mid m$ or $m$ is congruent to $n$ modulo $d$.

 \begin{prop}\label{pointorder m}
 Let $m$ be an integer such that
$$n<m<2n, \quad d |m.$$
  Let $f(x)\in K[x]$ be a degree $n$  polynomial without repeated roots.

a)  Assume that
$\CC_{f,d}$  has a point $P$ of order $m$
with $x(P)=a \in K$.
Suppose that $m<n+d$, i.e.,
$$m=d \cdot [(n+d)/d].$$

Then there exist a nonzero $A\in K$ and a polynomial $v(x) \in K[x]$ of degree $m/d$ such that
$$f(x)=A(x-a)^m+v(x)^d, \quad v(a) \ne 0, \quad P=(a,  v(a)).$$

b) 
Let $f(x)\in K[x]$ be a separable degree $n$ polynomial of the form
$$f(x)=A(x-a)^m+v(x)^d,$$ where
$$a \in K, \ A\in K, \ A\neq0, \quad v(x)\in K[x], \ \ \deg(v)=m/d, \ \ v(a) \ne 0.$$ 
Then
 $P=(a, v(a))$ is   a torsion $K$-point on  $\CC_{f,d}$
of order $m$.

In addition, if $K_0$ is a subfield of $K$ such that
$$A,a \in K_0, \quad v(x) \in K_0[x],$$
then $P \in \CC_{f,d}(K_0)$.
\end{prop}

 \begin{proof}
a)  Let $P$ be a $K$-point of order $m$ on $\CC_{f,d}$.   We have   \beq\label{order}m(P)-m(\infty)=m((P)-(\infty))=\div(h), \eeq where $h$ is a rational function on $\mathcal C_{f,d}$. The point $O$ is the   only pole of $h$ and it has multiplicity $m$. Therefore $h$ is a polynomial in $x,y$.
By \eqref{1} and \eqref{2},
the rational function $h$ can either be represented by a polynomial of the form $h=v(x)$ or
\beq \label{22} h=s_{\alpha_1}(x)y^{d-\alpha_1}+s_{\alpha_2}(x)y^{d-\alpha_2}+\cdots+s_{\alpha_k}(x)y^{d-\alpha_k}+v(x), \eeq
 where all polynomials $s_{\alpha_i}(x)$ are nonzero and all $\alpha_i$ are distinct nonnegative {\sl integers} that are strictly less than $d$.

 First assume that $h=v(x)$. Then
 $$m(P)-m(\infty)=\div(v).$$
The point $P$ is a zero of $v(x)$. It follows that all  points $\gamma P$
($\gamma \in \mu_d$) also are zeros of $v(x)$. Since $P$ is the   only zero of $h=v(x)$,
we have $\gamma P=P$ for all $\gamma$, which implies that $P=(w_j,0)$ for some $j$.
By Theorem \ref{thm0}, $P$ has order $d$, i.e., $d=m$, which is wrong, since $d<n<m$.

\vskip.3cm

 Now let
\beq \label{22} h=s_{\alpha_1}(x)y^{d-\alpha_1}+s_{\alpha_2}(x)y^{d-\alpha_2}+\cdots+s_{\alpha_k}(x)y^{d-\alpha_k}+v(x), \eeq
 where all polynomials $s_{\alpha_i}(x)$ are nonzero and all $\alpha_i$ are distinct nonnegative {\sl integers} that are strictly less than $d$.
By \eqref{3},
\beq\label{33}
m=\max\{d\deg (s_{\alpha_i})+n(d-\alpha_i),d\deg (v)\}.
\eeq
It follows from the inequality $m<n+d$  that  $\deg (s_{\alpha_i})=0$ for all $i$. Indeed, if $\deg (s_{\alpha_i})\geq1$ for some $i$, then
 \eqref{33} implies that
 $$m\geq d\deg (s_{\alpha_i})+n(d-\alpha_i)\geq n+d,$$ a contradiction.
If $k\geq 2$, then    $d-\alpha_i\geq2$ for at least one $i$, which implies that $m\geq 2n$, a contradiction. Therefore $k=1$ and $h=sy+v(x)$, where $s$ is a nonzero constant. It follows from \eqref{33} that
$$m=\max\{n,d\deg (v)\}.$$
Since $m>n$ we have
$$m=d\cdot\deg(v), \quad \deg(v)=\frac{m}{d}>\frac{n}{n}=1.$$
It follows that
$$\deg(v) \ge 2.$$
       We may and will assume that $h=v(x)-y$.
From \eqref{order} we get
$$m(P)-m(O)=\div(v(x)-y).$$
Applying Remark \ref{Norm} to $\phi=v(x)-y$, we conclude that
$$\prod_{\gamma \in \mu_d}( v(x)-\gamma y)=A^{\prime}(x-a)^m,$$
where $A^{\prime}$ is a nonzero element of $K$. This implies that
$$v(x)^d-f(x)=v(x)^d-y^d=A^{\prime}(x-a)^m,$$
i.e.,
$$v(x)^d-f(x)=A^{\prime}(x-a)^m, \quad f(x)=-A^{\prime}(x-a)^m+v(x)^d.$$
Hence, 
$$f(x)=A(x-a)^m+v(x)^d,$$
where $A=-A^{\prime}\in K$, $A\neq0$ and $v(x)$ is a polynomial of degree $m/d$.  Since $f(x)$ has no repeated roots, $v(a)\neq0$ (recall that $\deg(v) \ge 2$ and $d \ge 2$).
This proves (a).

b)
Let $f(x)\in K[x]$ be a degree $n$ separable polynomial of the form
$f(x)=A(x-a)^m+v(x)^d$, where $A\in K, A\neq0$ and $v(x)$ is a polynomial of degree  $m/d$.
 Let us put
 $$c:=v(a)\in K.$$ Then $P=(a,c)$ lies on the curve $y^d=A (x-a)^m+v(x)^d$. It follows from  $y^d- v(x)^d=A(x-a)^m$ that all zeros of the rational function $y-v(x)$ have abscissa $a$. But each $K$-point on $\CC_{f,d}$ with abscissa $a$
has the form $(a,\gamma c)$, where $\gamma \in \mu_d$.
Assume that for some $\gamma \ne 1$  the point $\gamma P=(a,\gamma c)$
is also a zero of  $y-v(x)$. Since $c=v(a)\neq0$, we obtain
$$0=\gamma c-v(a)=\gamma c-c=(\gamma-1)c\neq 0,$$ which gives us a contradiction. Therefore $y-v(x)$ has only one zero, namely $P$.
Since $y-v(x)$ has the {\sl only} pole at $O$ and its multiplicity is $m$ (because $n<m$), we have
$$m(P)-m(O)=\div(y-v(x)),$$ and so $P$ has order $m_1>1$ that divides $m$.

Since $y(P)=v(a) \ne 0$ and $n<m<2n$ by assumption,  it follows from Corollary \ref{O2n} that $P$ has order $m$.

In order to finish the proof, it remains to notice that $c=v(a)\in K_0$, because $a \in K_0$ and $v(x)\in K_0[x]$.
\end{proof}

The existence of degree $n$ polynomials of the form $A (x-a)^m+v(x)^d$ without repeated roots will be discussed in the next section.

\section{Polynomials of the form $A (x-a)^m+v(x)^d$}\label{separ}
\vskip.5cm
Throughout this section, $m>n$ is an integer and $\ell:=m/d$. 


\begin{thm}\label{exist order m}

\begin{itemize}
 \item[(a)]
 If $m= d \cdot [(n+d)/d]$ is
  $(n,d)$-reachable over $K$ then
 $$n-m+m/d=n-(d-1)  [(n+d)/d] \ge 0.$$

 \item[(b)]
 Suppose that $d |m$, i.e., $\ell=m/d$ is an integer,
  $$n-m+\ell  \ge 0,$$
 $K_0$ is an infinite subfield of $K$,
 and
 $\fchar(K)=\fchar(K_0)$ enjoys one of the following properties.

 \begin{itemize}
  \item[(i)]
  $\fchar(K)=0$;
  \item[(ii)]
  if $n-m +\ell=0$, then $\fchar(K)$ does not divide
  $\ell$;
  \item[(iii)]
  if $n-m +\ell>0$, then $\fchar(K)$ does not divide
  $n-m +\ell$.
 \end{itemize}

 Then $m$ is $(n,d)$-reachable over $K_0$.

\end{itemize}
\end{thm}

\begin{Rem}
\begin{enumerate}
\item
Suppose that $d|m$. Then the integer
$$\ell=\frac{m}{d}>\frac{n}{d}>1,$$
hence, $\ell \ge 2$.  On the other hand,
 $$n-m +\ell=\ell-(m-n)<\ell,$$
 i.e.,
 $$n-m+\ell<\ell.$$
\item
Assume additionally that $n-m +\ell\ge 0$, i.e.,
$$0 \le n-m +\ell=n-m+m/d=n-(1-1/d)m,$$
i.e.,
$$m\le \frac{d}{d-1} n \le 2n.$$
It follows that
$$\ell =\frac{m}{d} \le \frac{2n}{2}=n.$$
We get
\begin{equation}
\label{lnm}
0 \le n-m+\ell<\ell \le n.
\end{equation}
It follows from Theorem  \ref{exist order m}(b) that $m$ is $(n,d)$-reachable over $K$
if either $\fchar(K)=0$ or $\fchar(K)>n$. This means that the third assertion of Main Theorem
follows from Theorem \ref{exist order m}(b).

\end{enumerate}

\end{Rem}

The proof of Theorem \ref{exist order m} is an almost immediate corollary
of the following two assertions.

\begin{prop}\label{deg n}
Let $m>n$ be an integer.
 Let $a, A$ be elements of $K$ and $A \ne 0$. Choose $B \in K$ such that $A=-B^d$.

a) Suppose that $v(x) \in K[x]$ is a polynomial such that the polynomial  $h(x)=A(x-a)^m+v(x)^d$ has degree $n$.

Then $d$ divides $m$, $\deg(v)=m/d=\ell$,   and
 the leading coefficient $B$ of $v(x)$ satisfies $B^d=-A$. In addition,  $v(x)=B x^{\ell}+q(x)$, where $q(x)\in K[x]$ is a nonzero polynomial  whose degree $\deg(q)$ is
$n-m+\ell$. In particular,
$$0 \le \deg(q)= n-m+\ell=n-m+m/d,$$
i.e.,
$$n \ge  m (1-1/d).$$

b) Assume that $d$ divides $m$ and $n-m+\ell\geq0$
where $\ell=m/d$. Let $q(x)\in K[x]$ be an arbitrary nonzero polynomial of degree
$n-m+\ell$. Let us choose $B \in K$ such that $A=-B^d$ and put
 $v(x):=Bx^{\ell}+q(x)$. Then
 $$\deg(v)=\ell \ge 2$$
 and the polynomial
$h(x)=Ax^m+v(x)^d$
has degree $n$.
\end{prop}

\begin{prop}\label{sep}
Suppose that
 $d$ divides $m$ and  $n-m+\ell\geq0$. Assume also that $\fchar(K)$  enjoys one of the following properties.

 \begin{itemize}
  \item
  $\fchar(K)=0$;
  \item
  if $n-m +\ell=0$, then $\fchar(K)$ does not divide
  $\ell$;
  \item
  if $n-m +\ell>0$, then $\fchar(K)$ does not divide
  $n-m +\ell$.
 \end{itemize}

 Let $A$ and $D$ be  nonzero elements of $K$. Choose $B \in K$ such that
 $B^d=-A$.

 For each  nonzero $C \in K$  with $C \ne -D$ let us consider the  polynomials
 $$q_C(x)=Dx^{n-m+\ell}+C, \quad v_C(x)=B x^{\ell}+q_C(x)=B x^{\ell}+Dx^{n-m+\ell}+C.$$
 Then:
 \begin{itemize}
  \item[(i)] The polynomial
$$f_C(x):=A x^m+v_C(x)^d=Ax^m+\left(Bx^{\ell}+q_C(x)\right)^d \in K[x]$$
has  degree $n$.
\begin{itemize}
\item[(a)]
 If $n-m+\ell>0$ then the leading coefficient of  $f_C(x)$ is $dB^{d-1}D$;
in particular, $f_C(x)$ is monic if and only if $D=1/(dB^{d-1})$.
\item[(b)]
If  $n-m+\ell=0$ then the leading coefficient of $f_C(x)$ is $dB^{d-1}(D+C)$;
in particular, $f_C(x)$ is monic if and only if $D+C=1/(dB^{d-1})$.
\end{itemize}
\item[(ii)]
For all but finitely many $C\in K$ the polynomial $f_C(x)$ has no repeated roots.
\item[(iii)]
Let $K_0$ be an infinite subfield of $K$. Let us put
$$A:=-1, \ B:=1, \ D:=1  \in K_0\subset K.$$
 Then
there exist infinitely many $C \in K_0$ such that
$C \ne 0,-1$ and the corresponding polynomial
\begin{equation}
\label{fCk0}
f_C(x)=-x^m+ \left(x^{\ell}+x^{n-m+\ell}+C\right)^d \in K_0[x]\subset K[x]
\end{equation}
has degree $n$ and has no repeated roots. Its leading coefficient is $d$ if $n-m+\ell>0$
and $d(1+C)$ if $n-m+\ell=0$.
Last but not least:
$P=(0,C)$ is a torsion $K_0$-point on $\CC_{f_C,d}$ of order $m$.
 \end{itemize}
\end{prop}

\begin{proof}[Proof of Theorem \ref{exist order m}]
First assume that $m= d \cdot [(n+d)/d]$.
Suppose that $\CC_{f,d}(K)$  has a point  of order $m$ for some $f(x)$. Then it follows from Proposition \ref{pointorder m}(a) that
$d$ divides $m$ and $f(x)$ can be represented in the form
$$f(x)=A(x-a)^m+v(x)^d,$$ where $A\in K, A\neq0$, and $v(x)$ is a polynomial of degree  $\ell$.
It follows from Proposition \ref{deg n}(a)  that $n-m+\ell\geq0$. This proves Theorem \ref{exist order m}(a)
and the second assertion of Main Theorem.

Now assume that $m>n$ is an integer that is divisible by $d$,
 $$n-m+\ell\geq0,$$
   and $K_0$ and $\fchar(K)$ is as in Theorem \ref{exist order m}(b). It follows from
   Proposition \ref{sep}(iii) that $m$ is $(n,d)$-reachable over $K_0$.

\end{proof}
\begin{proof}[Proof of Proposition \ref{deg n}]
Let $h(x)$ be a polynomial of degree  $n$ and $a\in K$.
Assume that $h(x)$ can be represented in the form
$$h(x)=A(x-a)^m+v(x)^d,$$ where $v(x)\in K[x]$ is a polynomial.
Replacing $h(x)$ and $v(x)$ by $h(x+a)$ and $v(x+a)$ respectively, we may and will assume that $a=0$, i.e.,
$$h(x)=A x^m +v(x)^d.$$
Since $$\deg(h)=n<m=\deg(Ax^m),$$ the degrees of $Ax^m$ and $v(x)^d$ do coincide, and the leading coefficient of $v(x)^d$ is $-A$. This means that
$$d|m, \quad \ell =\deg(v),$$
and the leading coefficient $B$ of $v(x)$ satisfies $B^d=-A$. Let us write $v(x)$ in the form
$$v(x)=Bx^{\ell}+q(x),$$
where $\deg(q)<\ell$ and $B\neq0$. We have

$$-h(x)=-(A x^m +v(x)^d)=-A x^m-v(x)^d=\left(B x^{\ell}\right)^d-v(x)^d=$$
$$\left(B x^{\ell}-v(x)\right) \prod_{\gamma \in \mu_d, \ \gamma \ne 1}\left(B x^{\ell}-\gamma v(x)\right)= - q(x) \prod_{\gamma \in \mu_d, \ \gamma \ne 1}\left(B x^{\ell}-\gamma v(x)\right),$$
which means that
$$h(x)=q(x) \prod_{\gamma \in \mu_d, \ \gamma \ne 1}\left(B x^{\ell}-\gamma v(x)\right).$$
The degree of the first factor in the RHS equals $\deg(q)$ while all the other $(d-1)$ factors have degree $\ell$ and their coefficient  at $x^{\ell}$
is
$$B -\gamma B=B(1-\gamma)\neq 0.$$
 Since the degree of the LHS is $n$, we get the equality
$$n=\deg(q)+(d-1)\ell=\deg(q)+d\ell-\ell=\deg(q)+m-\ell,$$
i.e.,
$$\deg(q)=n-m+\ell,$$
which ends the proof of assertion (a).

Now assume that $d|m$ (i.e., $\ell$ is an integer) and $n-m+\ell\geq0$. Choose $B\in K$  such that $B^d=-A$, take an arbitrary nonzero
polynomial $q(x)\in K[x]$ of degree $n-m+\ell$ and consider the polynomial $v(x):=Bx^{\ell}+q(x)$.
Then $\deg(v)=\ell$, because
$$\deg(B x^{\ell})-\deg(q)=\ell -(n-m+\ell)=m-n>0.$$
On the other hand,
$$\ell=\frac{m}{d}>\frac{n}{n}=1$$
and therefore the integer $\ell>1$, i.e., $\ell \ge 2$.

Let us prove that  the polynomial $$h(x)= Ax^m+v(x)^d$$ has degree $n$.
We have (as above)
$$-h(x)=-(A x^m +v(x)^d)=-A x^m-v(x)^d=\left(B x^{\ell}\right)^d-v(x)^d=$$
$$\left(B x^{\ell}-v(x)\right) \prod_{\gamma \in \mu_d, \ \gamma \ne 1}\left(B x^{\ell}-\gamma v(x)\right)= - q(x) \prod_{\gamma \in \mu_d, \ \gamma \ne 1}\left(B x^{\ell}-\gamma v(x)\right),$$
which means that
\begin{equation}
 \label{Anvq}
h(x)=q(x) \prod_{\gamma \in \mu_d, \ \gamma \ne 1}\left(B x^{\ell}-\gamma v(x)\right).
\end{equation}
The degree of the first factor in the RHS equals $n-m+\ell$, while all the other $(d-1)$ factors have degree $\ell$ (their coefficient at $x^{\ell}$
is $B -\gamma B=B(1-\gamma)\neq0$). This implies that
$$\deg(h)=(n-m+\ell)+(d-1)\ell=n-m+\ell d=n-m+m=n.$$
 This ends the proof of assertion (b).
\end{proof}
\begin{proof}[Proof of Proposition~\ref{sep}]
Let  us consider
polynomials
$$q_C(x)=Dx^{n-m+\ell}+C,$$ where
$$C\in K, \quad  C\neq0,-D.$$
Our conditions on $D$ and $C$ imply that $q_C(x)$ is either the nonzero constant  $D+C$ if $n-m+\ell=0$,
or a polynomial of positive degree $n-m+\ell$ with leading coeffcient $D$ if $n-m+\ell>0$. In both cases
$$\deg(q_C)=n-m+\ell.$$
We need to show that for all but finitely  many $C$ the polynomial
$f_C(x)=Ax^m+\left(Bx^{\ell}+q_C(x)\right)^d$
has degree $n$ and has
no repeated roots. Using the same calculation as in the proof of \eqref{Anvq} for $h(x)=f_C(x)$ and
$$q(x)=q_C(x)=D x^{n-m+\ell}+C,$$ we get
$$f_C(x)=q_C(x) \prod_{\gamma \in \mu_d, \ \gamma \ne 1}\left(B x^{\ell}-\gamma v(x)\right)=$$
$$q_C(x) \prod_{\gamma \in \mu_d, \ \gamma \ne 1}
\left(B(1-\gamma)x^{\ell}-\gamma q_C(x)\right)
=q_C(x) \prod_{\gamma \in \mu_d, \ \gamma \ne 1} h_{\gamma,C}(x),$$
where
$$h_{\gamma,C}(x)=B(1-\gamma)x^{\ell}-\gamma q_C(x)=B(1-\gamma)x^{\ell}-\gamma D x^{n-m+\ell}-\gamma C$$
is a  polynomial of  degree $\ell\ge 2$ with leading coefficient $B(1-\gamma)$, because
$$\ell-(n-m+\ell)=m-n>0, \quad \ell=\frac{m}{d}>\frac{n}{n}=1.$$
This implies that
$$\deg(f_C)=(n-m+\ell)+(d-1)\ell=n.$$
Recall that the leading coefficient of $q_C(x)$ is $D$ when $n-m+\ell>0$; if $n-m+\ell=0$ then the leading coefficient is $D+C$.
Since the leading coefficient of a polynomial $h_{\gamma,C}(x)$ is $B(1-\gamma)$,  the leading coefficient of their product
$\prod_{\gamma \in \mu_d, \ \gamma \ne 1}h_{\gamma,C}(x)$ is
$$\prod_{\gamma \in \mu_d, \ \gamma \ne 1}B(1-\gamma)=B^{d-1}d.$$
This implies that the leading coeffient of $f_C(x)$ equals to  $dB^{d-1}D$ if $n-m+\ell>0$ and to $dB^{d-1}(D+C)$ if $n-m+\ell=0$ .
This proves (i).

Let us prove (ii).  First, notice that the constant term of $q_C(x)$  is $C$  while the constant term of  $h_{\gamma,C}(x)$ is $-\gamma C$; all  other coefficients of $q_C$ and $h_{\gamma,C}$  do {\sl not} depend on $C$. This implies that each $\alpha \in K$ may be a root of $f_C(x)$ only for finitely many $C\in K$.

Our conditions on $\fchar(K)$
imply that the derivatives of
all  $h_{\gamma,C}(x)$ are nonconstant polynomials
that do not depend on $C$.
 On the other hand, the conditions on $\fchar(K)$ and $C$ imply that either $q_C(x)$ is a nonzero constant $D+C$ (if $n-m+\ell=0$) or its derivative
 $(n-m+\ell) x^{n-m+\ell-1}$
 is a nonzero polynomial that does {\sl not} depend on $C$. So, the union $U$ of the set of repeated roots of $q_C(x)$ and the sets of repeated roots of all $h_{\gamma,C}(x)$
 lies in a certain
  finite subset of $K$ that does not depend on $C$. This implies the finiteness of  the set
 $S$ of such  $C \in K\setminus \{0,-D\}$ such that either $q_C(x)$ or $h_{\gamma,C}(x)$ (for some $\gamma$) has a repeated root.
 Hence, for all  $C \in K\setminus (\{0,-D\}\cup S)$ the polynomial $q_C(x)$  and all $h_{\gamma,C}(x)$ have no repeated roots. So, the only chance for $f_C(x)$ to have a repeated root for such $C$  is
when  either $q_C(x)$ and $h_{\gamma,C}(x)$ have a common root for some $\gamma$ or $h_{\gamma_1,C}(x)$ and $h_{\gamma_2,C}(x)$ have a common root for distinct $\gamma_1,\gamma_2 \in \mu_d \setminus \{1\}$.
Such a common root could not be $0$, because the constant terms of $q_C(x)$ and all $h_{\gamma,C}(x)$ are not $0$.

Suppose $\alpha\in K$ is a common root of
$h_{\gamma_1,C}(x)$ and $h_{\gamma_2,C}(x)$, where $\gamma_1 \ne \gamma_2$. Then $\alpha$ is a root of the polynomial
$$\frac{1}{\gamma_1}h_{\gamma_1,C}(x)-\frac{1}{\gamma_2}h_{\gamma_2,C}(x)=$$

$$\left(B\frac{1-\gamma_1}{\gamma_1}x^{\ell}-\frac{\gamma_1}{\gamma_1} D x^{n-m+\ell}-\frac{\gamma_1}{\gamma_1} C\right) -\left(B \frac{1-\gamma_2}{\gamma_2} x^{\ell}-\frac{\gamma_2}{\gamma_2} D x^{n-m+\ell}-\frac{\gamma_2}{\gamma_2} C\right)$$
$$=B\left(\frac{1}{\gamma_1}-\frac{1}{\gamma_2}\right)x^{\ell},$$
whose only root is $0$, which could not be the case. So,
$h_{\gamma_1,C}(x)$ and $h_{\gamma_2,C}(x)$ have no common roots for distinct $\gamma_1$ and $\gamma_2$.

Suppose that $\beta \in K$ is a common root of $q_C(x)$ and $h_{\gamma,C}(x)$ for some $\gamma \in \mu_d\setminus \{1\}$. Then $\beta$ is a root of
$$\begin{aligned}\gamma q_C(x)+h_{\gamma,C}(x)&=\gamma(Dx^{n-m+\ell}+C)+(B(1-\gamma)x^{\ell}-\gamma D x^{n-m+\ell}-\gamma C)\\&=B(1-\gamma)x^{\ell},\end{aligned}$$
whose only root is $0$, which is not the case. Hence, $q_C(x)$ and $h_{\gamma,C}(x)$ have no common roots.
This proves that $f_C(x)$ has no repeated roots for all but finitely many $C \in K$. This proves (ii).

It remains to prove (iii). Since $K_0$ is infinite, it follows from the  already proven assertions (i) and (ii) that there  exist infinitely many ``nice'' $C \in K_0 \setminus \{0,-1\}$ such that the polynomial
$f_C(x)\in K_0[x]$ defined in \eqref{fCk0} has degree $n$ and the desired leading coefficient but  has no repeated roots. Applying Proposition \ref{pointorder m}(b)
to $$A=-1, \  a=0, \ v(x)=x^{\ell}+x^{n-m+\ell}+C\in K_0[x]$$
with ``nice'' $C$,  we conclude that $P=(0,C)$ is a point of order $m$ in $\CC_{f_C,d}(K)$.
\end{proof}

\section{Points of order $n+ed$}
\label{order n+d points}
Throughout this section,  $K$ is a field of characteristic zero. 
 Recall   that $d<n$ and $(n,d)=1$. In this section we deal with torsion points  whose order is congruent to $n$ modulo $d$.
\begin{thm}
\label{nPluSd}
Suppose that $\fchar(K)=0$ and
 $f(x)\in K[x]$ is a monic degree $n$ polynomial without repeated roots.
Let $a$ be an element of $K$. Then the following conditions are equivalent.

\begin{itemize}
 \item[(a)]
There exists a $K$-point $P$ of $\CC_{f,d}$  of order $m_1=n+d$
with $x(P)=a$.
\item[(b)]
There exist $a_1 \in K$ and a polynomial $v(x) \in K[x]$
such that
$$a_1 \ne a, \quad d-1 \le \deg(v)<\frac{m_1}{d}=\frac{n+d}{d}, \quad v(a_1)\ne 0,$$
and
$$f(x)=\frac{(x-a)^{m_1}-v(x)^d}{(x-a_1)^d}.$$
\end{itemize}
If these equivalent conditions hold, then
$m_1=n+d >d^2-d$, i.e.,  $n>d^2-2d$.
\end{thm}
\begin{proof}
Suppose that
$P=(a,b) \in \CC_{f,d}(K)$ has order $m_1$.  Then
$$m_1(P)-m_1(O)=\div(\vf),$$
where $\vf$ is a non-constant rational function on $\CC_{f,d}$. We know (Section \ref{functions})
 that either $\vf=v(x)\in K[x]$
or
$\vf$ has the form $$\vf=s_{\alpha_1}(x)y^{d-\alpha_1}+s_{\alpha_2}(x)y^{d-\alpha_2}+\cdots+s_{\alpha_k}(x)y^{d-\alpha_k}+v(x),$$
where
all polynomials $s_{\alpha_i}(x)$ are nonzero and all $\alpha_i$ are distinct positive less than $d$. The order of the pole of $v(x)$ at $O$
is $d\deg(v)$ and the order  of the pole of $\vf$ at $O$ is $m_1$. So the equality $\vf=v(x)$ would imply that $m_1=d\deg v(x)$, which is impossible since
$m_1=n+d$ and $(n,d)=1$.
Since $n<n+d<2n$, we can write $\vf$ in the form
$$\vf=u(x)y-\mu v(x),$$ where $\mu \in K$ is a   $d$th root of $-1$,
and
$$u(x),\ v(x)\in K[x]; \quad
\deg(u)=1, \ \deg(v)<m_1/d.$$
Since $\deg(u)=1$, the polynomial $u(x)$ has precisely one root, say, $a_1$ and coincides with $\lambda(x-a_1)$ for some {\sl nonzero} $\lambda \in K$.

We have
$$0=\phi(P)=u(a)b-\mu v(a).$$
So, if $a_1=a$, then $u(a)=0$ and therefore $v(a)=0$, i.e., $v(x)=(x-a)w(x)$
for some polynomial $w(x) \in  K[x]$. This implies that $\phi=u(x)y-\mu v(x)$ vanishes at all $d$ points with abscissa $a$, which is wrong. So,
$a \ne a_1.$

Applying Remark \ref{Norm} to $\phi=u(x)y-  \mu v(x)$, we conclude that
 the polynomial
$$u(x)^d f(x)+v(x)^d=
\prod_{\gamma \in \mu_d}(u(x)y- \gamma \cdot \mu v(x))$$
coincides with $A(x-a)^m$ for some nonzero   $A\in K$.
 Since $\vf$ is defined up to a nonzero constant factor, we may and will assume that
 \beq\label{f(x)1} u(x)^d f(x)+v(x)^d =(x-a)^{m_1} \eeq
 This means that
 $$\alpha^d (x-a_1)^d f(x)+v(x)^d =(x-a)^{m_1}.$$
 Comparing the leading coefficients of both sides and taking into account that $f(x)$ is monic, we obtain that $\alpha^d=1$ and therefore
 $$ (x-a_1)^d f(x)+v(x)^d =(x-a)^{m_1}.$$

It follows that the polynomial
$(x-a)^{m_0}-v(x)^d$ of degree $m_1$
is divisible by $(x-a_1)^d$.
In particular, $v(a_1)^d=(a_1-a)^{m_1} \ne 0$.  Hence,
$v(a_1) \ne 0$.
It remains  to check that $\deg(v)\ge d-1$. In order to do that, let us put
$$w(x):=v(x+a_1) \in K[x].$$
Then
$$w(0) \ne 0, \quad  \deg(w)=\deg(v),$$
and
$(x-a+a_1)^{m_1}-w(x)^d$ is divisible by $x^d$. The divisibility condition means that
$$\left(1+\frac{x}{(-a+a_1)}\right)^{m_1}-\frac{w(x)^d}{(-a+a_1)^{m_1}}=
\left(1+\frac{x}{(-a+a_1)}\right)^{m_1}-\left(\frac{w(x)}{(-a+a_1)^{m_1/d}}\right)^d$$
is divisible by $x^d$. Let us introduce a new variable
$$z=\frac{x}{-a+a_1}$$
and consider a polynomial
$$W(z):=\frac{w\left((-a+a_1)z\right)}{(-a+a_1)^{m_1/d}} \in K[z].$$
Obviously,
$$W(0)=\frac{w(0)}{(-a+a_1)^{m_1/d}} =\frac{v(a_1)}{(-a+a_1)^{m_1/d}}  \ne 0, \quad \deg(W)=\deg(w)=\deg(v),$$
and $z^{m_1}-W(z)^d$ is divisible by $z^d$ in the ring of polynomials $K[z]$.
So, it suffices to check that
$$\deg(W) \ge d-1.$$

Indeed,
 the polynomial
$$ ((1+z)^{m_1/d})^d-W(z)^d$$ is divisible by $z^d$ in $K[z]$ and therefore
in the ring  of power series $K[[z]]$ as well. Since
\beq\label{root} ((1+z)^{m_1/d})^d-W(z)^d=\prod_{\gamma \in \mu_d}((1+z)^{m_1/d}-\gamma W(z)),\eeq
 and only the first factor on the right-hand side of \eqref{root} is divisible by $z$, we obtain that this factor is divisible by $z^d$. Hence,
$$(1+z)^{m_1/d}\equiv W(z)\mod z^d.$$
Since $m_1/d$ is {\sl not} an integer,  all {\sl generalized binomial} coefficients
$\binom{m_1/d}{k}$ in the expansion of $(1+z)^{m_1/d}$ are nonzero rational numbers. This implies that all the coefficients
of $W(z)$ up to the $(d-1)$th power are nonzero, i.e., the degree of $W(z)$ is at least $d-1$. Hence,
$$\deg(v)=\deg(W) \ge d-1.$$
Since
$$d-1 \le \deg(v)<m_1/d,$$ we get
$$d^2-d<m_1=n+d.$$
\end{proof}

\begin{prop}\label{point of order m}
Let $e$ be a positive integer and $m:=n+ed$.
Assume that
$$n+ed>d(de-1).$$ 
Let $v(x) \in K[x]$ be a polynomial that enjoys the following properties.
\begin{itemize}
 \item
 $\deg(v)<m/d=(n+ed)/d$.
 \item
 The polynomial $(1+x)^m-v(x)^d$ is monic, has distinct nonzero roots and  root $0$ of multiplicity $ed$ or $ed+1$.
\end{itemize}

Then the following statements hold.
\begin{itemize}
\item[(i)]
$$f(x):=\frac{(1+x)^m-v(x)^d}{x^{ed}}$$ is a degree $n$ polynomial without repeated roots and the corresponding $\CC_{f,d}(K)$ has a torsion point $P$  with abscissa $-1$,
whose order is $m$.
\item[(ii)]
Suppose that all the coefficients of $v(x)$ lie in a subfield $K_0$ of $K$.  Then
$$-f(x) \in K_0[x]\subset K[x], \ \deg(-f)=n,$$
the polynomial $-f(x)$ has no repeated roots, and its leading coefficient is $-1$.

In addition,  $Q=\left(-1,(-1)^e v(-1))\right)$ is a torsion $K_0$-point of order $m$ on the curve $\CC_{f,d}$.
\end{itemize}
\end{prop}

\begin{proof}
The proof follows word for word a similar statement in Sec.\ref{W^1}, Proposition \ref{pointorder m}. Indeed,
assume that the polynomial $(1+x)^m-v(x)^d$ is divisible by $x^{ed}$ and has distinct nonzero roots while $0$ is the root of multiplicity $ed$ or $ed+1$.
Then
$$f(x)=\frac{(1+x)^m-v(x)^d}{x^{ed}}$$
is a polynomial without repeated roots whose degree is
$$\deg(f)=\deg\left((1+x)^m-v(x)^d\right)-ed.$$
On the other hand,
$$\deg\left(v(x)^d\right)=d \cdot \deg(v) <d \cdot  \frac{m}{d}=m=\deg((1+x)^m).$$
Hence,
$$\deg\left((1+x)^m-v(x)^d\right)=\deg\left((1+x)^m\right)=m,$$
and the leading coefficient of $(1+x)^m-v(x)^d$ is the same as that of $(1+x)^m$, i.e., is
$1$.
Therefore
$$\deg(f)=m-ed=(n+ed)-ed=n$$
and the leading coefficient of $f(x)$ is $1$, i.e., $f(x)$ is monic.
We have
\begin{equation}
\label{f1L}
x^{ed}f (x)+v^d(x)=(1+x)^m.
\end{equation}
Choose $\lambda \in K$ such that
$$\lambda^d=-1$$ and let us put
 $$c=\lambda (-1)^e v(-1).$$
Then the point $P=(-1,c)$ lies on the curve $y^d=f(x)$, since it follows from \eqref{f1L} that
$$(-1)^{ed} f(-1)=0^m-v(-1)^d=-v(-1)^d=\lambda^d v(-1)^d$$
and therefore
$$f(-1)=(-1)^{ed} \lambda^d v(-1)^d=\left( (-1)^e \lambda v(-1)\right)^d=c^d.$$
In addition, $P$ is a zero of  the rational function $\phi:=x^e y-\lambda v(x)$, because
$$(-1)^e \cdot \lambda (-1)^e v(-1)-\lambda v(-1)=\lambda v(-1)-\lambda v(-1)=0.$$
We will prove that $P$ has order $m$ (then we get $d$ points of order $m$, namely, all the points
$\gamma P=(a,\gamma c), \; \gamma \in \mu_d$).
It follows from
$$(x+1)^m=x^{ed}y^d+ v(x)^d=\prod_{\gamma \in \mu_d}\left(x^ey-\gamma \lambda v(x)\right)$$
that all zeros of  $x^ey-\lambda v(x)=\phi$ have abscissa $-1$. But each point on the curve with abscissa $-1$
has the form $(-1,\gamma c)$ with $\gamma \in \mu_d$. Assume that for some $\gamma \ne 1$ the point $\gamma P=(-1,\gamma c)$
is also a zero of  $x^ey-\lambda v(x)$. Then
$$\begin{aligned}0=\phi(\gamma P)=(-1)^e (\gamma  c)-\lambda v(-1)\\=(-1)^e \gamma (-1)^e \lambda v(-1)) -\lambda v(-1)=(\gamma-1)\lambda v(-1),\end{aligned}$$
hence,  $v(-1)=0$, which is wrong.
 Therefore, $x^ey-\lambda v(x)$
has a single zero, namely $P$.
Since $x^e y-\lambda v(x)$ has the only pole at $O$ and its multiplicity is $m$
(because  $O$ is the pole of  $x^e y$ of order $ed+n=m$ and the pole of $v(x)$ of order
$d \deg(v)<m$), we obtain that
$$m(P)-m(O)=\div\left(x^e y-\lambda v(x)\right)$$
is a principal divisor.
So $P$ has order $>1$ that divides $m$.
In order to apply Corollary \ref{O2n}, it suffices to check that either $m<2n$ or $m$ is odd and strictly less than $3n$. Let us do it.
Indeed, we have
$$m=n+ed>d(ed-1),$$
i.e.,
$$n>(d-1)de-d, \quad n+d>(d-1)ed, \quad ed<\frac{n+d}{d-1}.$$
If $d>2$, then
$$d \ge 3, \ d-1\ge 2, \quad ed < \frac{n+d}{d-1} < \frac{n+n}{2}=n,$$
and therefore
$$m=n+ed<n+n=2n.$$

If $d=2$, then
$$ed<\frac{n+d}{d-1}=n+2,$$ and therefore $ed \le n+1$. Since $(n,d)=1$, we get $ed \ne n$, so, either
$ed<n$ or $ed=n+1$. In the former case
$$m=n+ed<n+n=2n.$$
If $ed=n+1$, then
$$m=n+ed=n+(n+1)=2n+1$$
is an odd integer that is strictly less than $3n$. So, we are in position to apply
Corollary \ref{O2n}, which implies that $P$ has order $m$. This proves (i).

In order to prove (ii), let us consider the  $\mu_d$-equivariant $K$-isomorphism of curves
$$\Psi: \CC_{-f,d}\to \CC_{f,d}, \quad  (x,y) \mapsto (x, \lambda y),$$
which sends $O_{-f,d}$ to $O_{f,d}$ and
$$Q=\left(-1,(-1)^e v(-1))\right)\in \CC_{-f,d}(K_0)\subset  \CC_{-f,d}(K)$$
 to
$$P=\left(-1, \lambda (-1)^e v(-1))\right)\in \CC_{f,d}(K)$$
 (recall that $\lambda^d=-1$).
Clearly, $\Phi$ respects the orders of points. Since the order of $P$ is $m$,
the order of $Q$ is also $m$ and therefore $Q$ is a point of order $m$ in $\CC_{-f,d}(K_0)$.

\end{proof}

\section{Polynomial Algebra}
\label{powerSeries}
As above, $d \ge 2$ is an integer. Let $m$ be an integer such that $(d,m)=1$.
Let us denote the noninteger rational number $m/d$ by $r$.
Then  both $r$ and $r-1=(m-d)/d$ are  rational numbers that are {\sl not} integers. Hence, for every nonnegative integer $k$
the corresponding generalized binomial coefficients
$$\binom{r}{k}=\frac{r(r-1) \dots (r-k+1)}{k!} \quad \text{and} \quad \binom{r-1}{k}$$
are {\sl not} $0$.
 We will need Pascal's identity
\begin{equation}
\label{pascal}
\binom{r}{k}=\binom{r-1}{k}+\binom{r-1}{k-1}
\end{equation}
for all integers $k\ge 1$. Let us consider the formal power series
$$(1+x)^r=\sum_{k=0}^{\infty} \binom{r}{k} x^k, \quad (1+x)^{r-1}=\sum_{k=0}^{\infty} \binom{r-1}{k} x^k
\in K[[x]].$$
We will also need the following well known equalities in $K[[x]]$:
\begin{equation}
\label{rPLUS1}
\begin{aligned}
\left((1+x)^r\right)^d=(1+x)^{rd}=(1+x)^m, \\ \left((1+x)^{r-1}\right)^d=(1+x)^{(r-1)d}=(1+x)^{m-d},\\
 \quad (1+x)^{r}=(1+x) (1+x)^{r-1}. \end{aligned}
\end{equation}

 Let $E \ge 2$ be an integer and let us consider the ``truncation''  $V_{r,E}(x)$  of $(1+x)^r$ up to  first $E$ terms,
i.e.,
$$V_{r,E}(x) =\sum_{k=0}^{E-1}   \binom{r}{k}\cdot x^k  =1+rx+ \dots + \frac{r(r-1) \dots (r-E+2)}{(E-1)!} x^{E-1}.$$
Clearly, $V_{r,E}(x)$ is a polynomial with rational coefficients and
\begin{equation}
\label{degE}
\deg(V_{r,E})=E-1.
\end{equation}
It follows from \eqref{rPLUS1} combined with \eqref{pascal} that 
\begin{equation}
\label{VrE}\begin{aligned}
V_{r,E}(x)=(1+x)V_{r-1,E-1}(x)+ \left( \binom{r}{E}-  \binom{r-1}{E-1}\right) x^{E-1}\\
=(1+x)V_{r-1,E-1}(x)+ \binom{r-1}{E}x^{E-1}.
\end{aligned}\end{equation}
Hence $V_{r,E}(x)-(1+x)V_{r-1,E-1}(x)$ is a polynomial of degree $E-1$, because $\binom{r-1}{E} \ne 0$.

Taking into account that the derivative of $(1+x)^r$ equals $r (1+x)^{r-1}$, we obtain that
\begin{equation}
\label{Der}
V_{r,E}^{\prime}(x)=r V_{r-1,E-1}(x).
\end{equation}

 \begin{thm}\label{distinctroots}
 Suppose that
 $$m>d(E-1).$$
  Then the polynomial
 $$F(x)=F_{r,E}(x)=(1+x)^m - V_{r,E}(x)^d$$
  has degree $m$ and has no nonzero repeated roots. On the other hand, $0$ is a root of $F(x)$ with multiplicity $E$. In other words,
 $$f(x)=f_{r,E}(x):=\frac{F(x)}{x^E}=\frac{(1+x)^m - V_{r,d}(x)^d}{x^E}\in \mQ[x]\subset K[x]$$
 is a polynomial of positive degree $m-E$ without repeated roots and with nonzero constant term.
  \end{thm}
\begin{proof}
Taking into account that $E \ge 2$ and $d \ge 2$, we get
$$m-E>d(E-1)-E=(d-1)E-d \ge 2(d-1)-d=d-2 \ge 0,$$
which implies that
$$m-E>0, \quad \text{i.e.,} \quad m-E \ge 1.$$
We have
$$\deg(V_{r,E}^d)=d \deg(V_{r,E})=d(E-1)<m=\deg\left((1+x)^m\right).$$
This implies that
$$\deg(F)=\deg\left((1+x)^m\right)=m.$$
Suppose that   $F(x)$ has  a nonzero repeated root, say, $\beta$. Then $\beta$ is also a root of the derivative
\beq\label{Fder}\begin{aligned}F^{\prime}(x)&=m(1+x)^{m-1}-d V^{\prime}_{r,E}(x) V_{r,E}(x)^{d-1}\\ &=
m(1+x)^{m-1}-d   r \cdot V_{r-1,E-1}(x)  V_{r,E}(x)^{d-1}\\&= m(1+x)^{m-1}-m  V_{r-1,E-1}(x)  V_{r,E}(x)^{d-1}\\
&=m\left((1+x)^{m-1}-   V_{r-1,E-1}(x)  V_{r,E}(x)^{d-1}\right).\end{aligned}\eeq
Multiplying the last expression in \eqref{Fder} by $(1+x)/m$, we get that $\beta$ is a root of
$$H(x)=(1+x)^m -  (1+x) V_{r-1,E-1}(x) V_{r,E}(x)^{d-1}.$$
Hence, $\beta$ is a root of
$$\begin{aligned}&G(x)=H(x)-F(x)=V_{r,E}(x)^d-(1+x) V_{r-1,E-1}(x) V_{r,E}(x)^{d-1}=\\&V_{r,E}(x)^{d-1} \left(V_{r,E}(x)-(1+x) V_{r-1,E-1}(x)\right)
=  V_{r,E}(x)^{d-1}\binom{r-1}{E}x^{E-1} \neq 0 \end{aligned}$$
(here we used \eqref{VrE}).
Since $\beta\neq 0$, it
 is a root of $V_{r,E}(x)$.
 Since $\beta$ is a root of $F(x)$, it is also a root of $(1+x)^m$, i.e., $\beta=-1$.
We need to check that $-1$ cannot be a root of $V_{r,E}(x)$ and $0$ is a root of  $F_{r,E}(x)$ with multiplicity $E$.

 We have
 \begin{equation}
 \label{pAdic}
 V_{r,E}(-1)=\sum_{k=0}^{E-1}  \frac{r(r-1) \dots (r-k+1)}{k!}  \cdot (-1)^k.
 \end{equation}
  Let $p$ be a prime dividing $d$. Then $p$ does {\sl not} divide $m$, since $(m,d)=1$. Recall that $r=m/d$. It follows that
all terms in the RHS of \eqref{pAdic} have distinct $p$-adic orders that strictly decrease and are negative (except the first term). This implies that
$V_{r,E}(-1)$ is not a $p$-adic integer and therefore is not $0$. Hence, $-1$ is not a zero of $ V_{r,E}(x)$. The obtained contradiction proves
that $F(x)$ has no nonzero repeated roots.

It remains to check that  $F_{r,E}(x)$ has a zero of order $E$ at $0$. 
In order to prove it, we need to carry out computations in $K[[x]]$.  Namely,
$$(1+x)^m-V_{r,E}(x)^d=\left((1+x)^r\right)^d-V_{r,E}(x)^d=$$
$$\left((1+x)^r-V_{r,E}(x)\right)
\prod_{\gamma \in \mu_d, \ \gamma \ne 1}\left((1+x)^r-\gamma V_{r,E}(x)\right).$$

In this product the first factor is a power series that starts with (nonzero) term
$\binom{r}{E} x^E$,
while the other $(d-1)$ factors start with nonzero constant term $1-\gamma$ and therefore are invertible elements of $K[[x]]$. This implies that the product is divisible in $\mQ[[x]]$ by $x^E$ but not by $x^{E+1}$.
It follows that
$$(1+x)^m-V_{r,E}(x)^d=\left((1+x)^r\right)^d-V_{r,E}(x)^d$$
is divisible in $\mQ[[x]]$ by $x^E$ but not by $x^{E+1}$.
Taking into account that
$$K[x] \cap x^E K[[x]]=x^E K[x] \quad \text{ and } \quad K[x] \cap x^{E+1} K[[x]]=x^{E+1} K[x],$$
we conclude that
the polynomial   $(1+x)^m-V_{r,E}(x)^d$  is divisible in $K[x]$ by $x^E$ but not by $x^{E+1}$, which ends the proof.
\end{proof}
\section{Points of order $n+ed$ again}
\label{endMain}
\begin{proof}[Proof of Theorem \ref{mainT}(4)]
Combining Proposition \ref{point of order m} (applied to $e=1$ and $v(x)=V_{r,E}(x)$) and Theorem \ref{distinctroots} (applied to $E=d$), we obtain that
$(n+d)$ is $(n,d)$-reachable over $\mQ$  (and therefore over any field of characteristic zero) if $m=n+d>d^2-d$, i.e., $n>d^2-2d$. On the other hand, it follows from  Theorem \ref{nPluSd}
 that if $(n+d)$ is $(n,d)$-reachable over $K$ then 
$n>d^2-2d$.
\end{proof}
 \begin{proof}[Proof of Theorem \ref{mainT}(5)]
Combining Proposition \ref{point of order m} (applied to  $v(x)=V_{r,E}(x)$ and $K_0=\mQ$) and Theorem \ref{distinctroots} applied to $E=ed$, we obtain that
$m=n+ed$ is $(n,d)$-reachable over $\mQ$ (and therefore over any field of characteristic zero) if $n+ed>d(ed-1)$,
i.e., $n>ed^2-(e+1)d$.
\end{proof}


\begin{thebibliography}{99}

\bibitem{Arul} V. Arul (2020) {\sl Explicit Division and Torsion Points on Superelliptic Curves and Jacobians} [Doctoral dissertation, MIT].


\bibitem{BZ} B.M. Bekker, Yu.G. Zarhin, {\sl Torsion points of order ${2g+1}$ on odd degree hyperelliptic curves of genus $g$}.
Trans. Amer. Math. Soc. {\bf 373} (2020), 8059–8094.


\bibitem{BZR} B.M. Bekker, Yu.G. Zarhin, {\sl Torsion points of small order
on cyclic covers of $\mathbb P^1$}.  Ramanujan J.  {\bf 67}  (2025), article 68.
\bibitem{Box1} J. Boxall and D. Grant, {\sl Examples of torsion points on genus two curves}.  Trans. Amer. Math. Soc. {\bf 352} (2000),  no. 10, 4533--4555.

\bibitem{Box2} J. Boxall, D. Grant, and F. Lepr\'evost, {\sl 5-torsion points on curves of genus 2}. J. Lond. Math. Soc. (2) {\bf 64} (2001), 29--43.

\bibitem{Boxall} J. Boxall, {\sl
Bounds on the number of torsion points of given order on curves embedded in their Jacobians}.
J. Algebra {\bf 672} (2025), 145--176.

\bibitem{Flynn} E.V. Flynn, {\sl Sequences of rational torsions on abelian varieties}. Invent. Math. {\bf 106} (1991), 433--442.

\bibitem{Gal}S. D. Galbraith, S. M. Paulus, and N. P. Smart, {\sl Arithmetic on superelliptic curves}.  Mathematics of  Computation {\bf  71:237} (2000),  393--405.

\bibitem{Gendron} Q. Gendron, {\sl \'Equation de Pell-Abel et applications}.  Comptes Rendus Math. {\bf 360} (2022), 975--992.

\bibitem{Lep} F. Lepr\'evost, {\em Sur certain sous-groupes de torsion de jacobiennes de courbes hyperelliptiques de genre} $g \ge 1$.
Manuscripta Math. {\bf 92} (1997), 47--63.

\bibitem{PS} B. Poonen, M. Stoll, {\sl Most odd degree hyperelliptic curves have only one rational point}.
 Ann. of Math.  {\bf 180} (2014), 1137--1166.

\bibitem{Raynaud} M. Raynaud, {\sl Courbes sur une vari\'et\'e abeliennes et points de torsion}. Invent. Math. {\bf 71:1} (1983), 207--233.

\bibitem{Ribet} K. Ribet and Minhyong Kim, {Torsion points on modular curves and Galois Theory},                                                                            arXiv:math/0305281v1.
\bibitem{Tsermias} P. Tsermias, {\sl The Manin–Mumford Conjecture: A Brief Survey}. Bull.   Lond. Math. Soc., Vol. 32, Issue 6,  (2000),  641--652.

\bibitem{Zarhin} Yu. G. Zarhin, {\sl Division by $2$ on odd-degree hyperelliptic curves and their jacobians}.
Izvestiya Mathematics {\bf 83:3} (2019), 501--520.
\end{thebibliography}
\end{document}